\documentclass[final]{my-elsarticle}

\usepackage{amsfonts,amsmath,amssymb,amsthm,mathtools,mathrsfs}
\usepackage{paralist,enumitem}
\usepackage{hyperref,url}

\DeclareMathAlphabet{\mathpzc}{OT1}{pzc}{m}{it}

\newtheorem{theorem}{Theorem}[section]
\newtheorem{lemma}[theorem]{Lemma}
\newtheorem{proposition}[theorem]{Proposition}
\newtheorem{corollary}[theorem]{Corollary}

\newtheorem{definition}[theorem]{Definition}
\newtheorem{example}[theorem]{Example}

\newtheorem{remark}[theorem]{Remark}

\newcommand{\R}{\mathbb{R}}
\newcommand{\N}{\mathbb{N}}
\newcommand{\Z}{\mathbb{Z}}

\newcommand{\T}{\mathbb{T}}

\journal{Nonlinear Analysis: Hybrid Systems}

\begin{document}

\begin{frontmatter}
	
\title{Solutions of systems with the Caputo--Fabrizio fractional delta derivative on time scales}

\author[mymainaddress]{Dorota Mozyrska\corref{mycorrespondingauthor}}
\cortext[mycorrespondingauthor]{Corresponding author}
\ead{d.mozyrska@pb.edu.pl}

\author[mysecondaryaddress]{Delfim F. M. Torres}
\ead{delfim@ua.pt}

\author[mymainaddress]{Ma{\l}gorzata Wyrwas}
\ead{m.wyrwas@pb.edu.pl}

\address[mymainaddress]{Faculty of Computer Science, 
Bialystok University of Technology,\\ 
Wiejska 45a, 15-351 Bia{\l}ystok, Poland}

\address[mysecondaryaddress]{Center for Research and Development in Mathematics and Applications (CIDMA),
Department of Mathematics,University of Aveiro, 3810--193 Aveiro, Portugal}


\begin{abstract}
Caputo--Fabrizio fractional delta derivatives 
on an arbitrary time scale are presented. 
When the time scale is chosen to be the set 
of real numbers, then the Caputo--Fabrizio fractional 
derivative is recovered. For isolated or partly continuous 
and partly discrete, i.e., hybrid time scales, 
one gets new fractional operators. We concentrate 
on the behavior of solutions to initial value problems 
with the Caputo--Fabrizio fractional delta derivative 
on an arbitrary time scale. In particular, the exponential 
stability of linear systems is studied. A necessary and 
sufficient condition for the exponential stability of 
linear systems with the  Caputo--Fabrizio fractional 
delta derivative on time scales is presented. By considering 
a suitable fractional dynamic equation and the Laplace 
transform on time scales, we also propose a proper definition 
of Caputo--Fabrizio fractional integral on time scales. 
Finally, by using the Banach fixed point theorem, 
we prove existence and uniqueness of solution 
to a nonlinear initial value problem with 
the Caputo--Fabrizio fractional 
delta derivative on time scales.
\end{abstract}

\medskip

\end{frontmatter}


\section{Introduction}

There is a recent interest on the development
of fractional (noninteger) calculi on arbitrary nonempty
closed subsets of the real numbers. The subject was initiated
in 2011 by using the inverse Laplace transform on time scales
\cite{MyID:201}. The first local notion of fractional
derivative for functions defined on completely arbitrary time scales
seems to have been introduced in 2015 \cite{MyID:296}. Such results were
then extended to the nabla and nonsymmetric cases in \cite{MyID:320},
while the possibility of a fractional derivative to be
a complex number is addressed in \cite{MyID:330}.
In \cite{Aydogan2017,Aydogan2018,Baleanu_BVP2017}, the authors study the 
existence of solutions for fractional differential equations including 
the Caputo--Fabrizio derivative, that is, 
they concentrate on systems with continuous time.
Existence and uniqueness results for fractional initial value problems
on arbitrary time scales are investigated in \cite{MyID:328,MyID:365}
and the so-called conformable case is studied
in \cite{MyID:379,MyID:324}. Ortigueira et al.
introduced fractional derivatives on arbitrary time scales
through convolution \cite{MyID:343}. Here we are interested to extend
the interesting approach of Caputo--Fabrizio from the time scale
$\mathbb{T} = \mathbb{R}$ into an arbitrary time scale $\mathbb{T}$.
Our results are completely new in the discrete-time or hybrid cases.
Since hybrid systems are dynamical systems 
that exhibit both continuous and discrete dynamic behaviors, we decided to use
time scales, which are recognized as a useful tool to the study of 
continuous- and discrete-time systems. Observe that instead of describing the system 
by differential or difference equations, we consider here systems with delta 
and/or nabla derivatives to write systems defined on domains that are 
partly continuous and partly discrete.

The Caputo--Fabrizio fractional derivative,
which is the convolution of the exponential function
and the first order derivative, was introduced in 2015
in \cite{Caputo:Fabrizio}, see also \cite{CF1},
with the purpose of avoiding singular kernels.
The new notion is receiving an increasing of interest.
For example, in \cite{MR3787701} a numerical approach based
on optimization theory and Ritz's method is developed to deal with
Caputo--Fabrizio Fokker--Planck equations;
applications in linear viscoelasticity are presented in \cite{MR3776052};
while practical signal processing problems are investigated in \cite{MR3770784}.
Moreover, in \cite{TateishiRibeiroLenzi2017} the authors 
show that the Caputo--Fabrizio operator may be a simple and efficient way for 
incorporating different structural aspects into the system, opening new
possibilities for modeling and investigating anomalous diffusive processes. 
Another application of the Caputo--Fabrizio fractional derivative, to describe real-world 
problems, can be found for instance in \cite{ATANGANA2018688,Atangana2018,Gomez-Aguilar2017}.
To the best of our knowledge, we are the first to consider
Caputo--Fabrizio delta derivatives and integrals
on arbitrary time scales. In this way, 
we extend the deterministic definition.

The paper is organized as follows. In Section~\ref{sec:2},
we briefly review the necessary concepts from time scales.
Our results are then given in Section~\ref{sec:3}.
Firstly, we introduce Caputo--Fabrizio fractional derivatives on time scales
in Section~\ref{subsec:3.1}. Then, in Section~\ref{subsec:3.3}, by considering a suitable
fractional differential equation and the Laplace transform on time scales,
we arrive to a proper definition of Caputo--Fabrizio fractional
integral of order $\alpha$. Next, we study the stability of linear equations
with such time-scale fractional derivatives in Section~\ref{subsec:3.2}. Finally, by using
Banach's fixed point theorem, we prove in Section~\ref{subsec:3.4}
existence and uniqueness of solution to a nonlinear fractional
initial value problem on time scales. We end with Section~\ref{sec:conc}
of conclusions, including some possible directions of future research.


\section{Preliminaries on time scales}
\label{sec:2}

We recall here basic concepts and facts of the calculus on time
scales. For more information, the reader is referred
to~\cite{Bohner}.

A {\it{time scale}} $\mathbb{T}$ is an arbitrary nonempty closed
subset of the set of real numbers~$\mathbb{R}$. The standard
examples of time scales are $\mathbb{R}$, $h \mathbb{Z}$, $h>0$,
$\mathbb{N}$, $\mathbb{N}_0$, $2^{\mathbb{N}_0}$
or $\mathbb{P}_{a,b}= \bigcup\limits_{k=0}^ \infty[k(a+b), k(a+b)+a]$.
A time scale $\mathbb{T}$ is a topological space with the
relative topology induced from $\mathbb{R}$.
Let us now recall some basic operators defined on time scales.
The  {\it forward jump operator} $\sigma : \mathbb{T} \to
\mathbb{T}$ is  defined by \(\sigma (t):= \inf \{s \in \mathbb{T} :
s>t \}\) and $\sigma \left( \sup \mathbb{T} \right)=\sup
\mathbb{T}$, if $\sup \mathbb{T} \in \mathbb{T}$.
The {\it backward jump operator}  $\rho: \mathbb{T} \to \mathbb{T}$
is defined by \(\rho(t):= \sup \{s \in \mathbb{T} : s<t \}\)
and $\rho \left(\inf \mathbb{T} \right)=\inf \mathbb{T}$,
if \mbox{$\inf \mathbb{T} \in \mathbb{T}$}.
A point $t\in\mathbb{T}$ is called \emph{right-dense},
\emph{right-scattered}, \emph{left-dense} or \emph{left-scattered}
if $\sigma(t)=t$, $\sigma(t)>t$, $\rho(t)=t$,
and $\rho(t)<t$, respectively. We say that $t$ is \emph{isolated}
if $\rho(t)<t<\sigma(t)$, that $t$ is \emph{dense} if $\rho(t)=t=\sigma(t)$.
The {\it graininess function} $\mu: \mathbb{T}
\to [0, \infty )$ is defined by \(\mu(t):= \sigma(t)-t.\)

In order to define the delta derivative, one introduces the set
\begin{equation*}
\mathbb{T}^\kappa
:=
\begin{cases}
\mathbb{T}\setminus (\rho(\sup
\mathbb{T}),\sup \mathbb{T}],
& \text{if \ $\sup \mathbb{T} <\infty$},\\
\mathbb{T}, 
& \text{if \ $\sup \mathbb{T} =\infty$}.
\end{cases}
\end{equation*}

\begin{definition}
Let $f:\mathbb{T} \to \mathbb{R}$ and $t \in \mathbb{T}^{\kappa}$.
Then the number $f^\Delta (t)$, when it exists, with the property
that, for any $\varepsilon>0$, there exists a neighborhood $U$ of
$t$ such that
\begin{equation*}
|[f(\sigma(t))-f(s)]-f^\Delta(t)[\sigma(t)-s]|
\le \varepsilon|\sigma(t)-s|, \ \forall s \in U,
\end{equation*}
is called the \emph{delta derivative of $f$ at $t$}. The function
\mbox{$f^\Delta: \mathbb{T}^{\kappa} \to \mathbb{R}$} is called
the {\emph{delta derivative}} of $f$ on $\mathbb{T}^{\kappa}$.
We~say that $f$ is {\emph {delta differentiable}} on
$\mathbb{T}^\kappa$, if  $f^\Delta(t)$ exists for all $t \in
\mathbb{T}^\kappa$.
\end{definition}

\begin{remark}
Let us consider some illustrative time scales.
\begin{enumerate}
\item If $\mathbb{T}=\mathbb{R}$, then
$f:\mathbb{R}\to \mathbb{R}$ is delta differentiable at $t\in
\mathbb{R}$ if and only if \(f^{\Delta}(t)=\lim\limits_{s\to t}
\frac{f(t)-f(s)}{t-s}=f'(t)\), i.e., $f$ is differentiable in the
ordinary sense at $t$.
\item If $\mathbb{T}=\mathbb{Z}$, then
$f:\mathbb{Z}\to \mathbb{R}$ is always delta differentiable on
$\mathbb{Z}$ with
\(f^{\Delta}(t)=\frac{f(\sigma(t))-f(t)}{\mu(t)}=f(t+1)-f(t),\) for
all $t\in\mathbb{Z}$.
\item If $\mathbb{T}=\overline{q^\mathbb{Z}}$, then $f:\mathbb{T}\to
\mathbb{R}$ is always delta differentiable on
$\mathbb{T}\setminus\{0\}$ and
\(f^{\Delta}(t)=\frac{f(qt)-f(t)}{(q-1)t},\) for all $t\in
q^\mathbb{Z}$. Moreover,
$f^\Delta(0)=\lim\limits_{s\to 0}\frac{f(s)-f(0)}{s}$,
only if this limit exists.
\end{enumerate}
\end{remark}

\begin{definition}
A function $f:\mathbb{T} \to \mathbb{R}$ is called
{\emph{regulated}} if its right-side limits exist (finite) at all
right-dense points in $\mathbb{T}$ and its left-side limits exist
(finite) at all left-dense points in $\mathbb{T}$.
\end{definition}

\begin{definition}
A function $f:\mathbb{T} \to \mathbb{R}$ is called
{\emph{rd-continuous}} if it is continuous at the right-dense
points in $\mathbb{T}$ and its left-sided limits exist at all
left-dense points in $\mathbb{T}$.
\end{definition}

\begin{definition}
A continuous function $f:\mathbb{T}\to \mathbb{R}$ is
\emph{pre-differentiable} with $D$ (the region of
differentiation), if $D\subset \mathbb{T}^\kappa$,
$\mathbb{T}^\kappa\setminus D$ is countable and contains no
right-scattered elements of $\mathbb{T}$ and $f$ is differentiable
at each $t\in D$.
\end{definition}

If $f$ is regulated, then there exists a function $F$ that is
pre-differentiable with the region of differentiation $D$, such
that \( F^\Delta (t)=f(t)\) for all $t\in D$.
Any function $F$ that satisfies \( F^\Delta (t)=f(t)\) is called a
\emph{pre-antiderivative} of $f$. Then the \emph{indefinite
integral} of a regulated function $f$ is defined by
\(
\int f(t)\Delta t=F(t)+C,
\)
where $C$ is an arbitrary constant. The \emph{Cauchy integral} of a
regulated function $f$ is defined by
\(
\int\limits_{r}^{s} f(t)\Delta t=F(s)-F(r),
\)
for all $s,t\in \mathbb{T}$. A function $F:\mathbb{T}\to
\mathbb{R}$ is called an \emph{antiderivative of $f:\mathbb{T}\to
\mathbb{R}$} if it satisfies \( F^\Delta (t)=f(t), \)
for all $t\in \mathbb{T}^\kappa$.

It is known that every rd-continuous function has an
antiderivative, so it is delta-integrable \cite{Bohner}.
The set of rd-continuous functions on the time scale $\mathbb{T}$
is denoted by $C_{rd}(\mathbb{T})$.

\begin{example}
If $\mathbb{T}=\mathbb{R}$, then $\int\limits_{a}^{b} f(t)
\Delta t=\int\limits_{a}^{b} f(t) d t$, where the integral
on the right-hand side is the usual Riemann integral.
\end{example}

\begin{example}
If $\mathbb{T}=h\mathbb{Z}$, where $h>0$, then
$\int\limits_{a}^{b}  f(t) \Delta t
=\sum\limits_{k=\frac{a}{h}}^{\frac{b}{h}-1}
h\cdot f(kh)$, $a<b$.
\end{example}

Now, let us recall the concept of exponential function on a time scale.
The function $p: \mathbb{T}\to \mathbb{R}$ is called \emph{regressive},
if $1+p(t)\mu(t)\neq 0$ for all $t\in \mathbb{T}^{\kappa}$.
The set of all regressive and rd-continuous functions $f: \mathbb{T}
\to \mathbb{R}$ is denoted by $\mathcal{R}$.

Let us consider the following linear delta differential
equation on a given time scale $\mathbb{T}$:
\begin{equation}
\label{equation-linear}
x^{\Delta}(t)=p(t) x(t),
\end{equation}
where $t\in\mathbb{T}$ and $x(t)\in \mathbb{R}$.
If $p\in \mathcal{R}$, then one defines the
{\em exponential function $e_p(\cdot,t_0)$ on the time scale $\mathbb{T}$} by
\begin{equation*}
 e_p(t,t_0)=
\begin{cases}
\exp \left(\int_{t_0}^t p(s)ds\right),
&\text{for $t\in \mathbb{T}$, $\mu=0$;}\\
\exp \left(\int_{t_0}^t \frac{{\rm Log}
(1+\mu(s)p(s))}{\mu (s)}\Delta s\right),
&\text{for $t\in \mathbb{T}$, $\mu>0$,}
\end{cases}
\end{equation*}
where ${\rm Log}$ is the principal logarithm function.

\begin{theorem}[See \cite{Bohner}]
Let $t_0 \in \mathbb{T}$ and $x_0 \in \mathbb{R}$. Then,
system \eqref{equation-linear} with initial condition $x(t_0)=x_0$ has a
unique solution $x: [t_0,+\infty)\cap\mathbb{T} \rightarrow \mathbb{R}$.
\end{theorem}

\begin{theorem}[See \cite{Bohner}]
Let $t_0 \in \mathbb{T}$ and $p\in \mathcal{R}$. Then $e_p(\cdot,t_0)$
is the solution of the initial value problem
\begin{equation*}
x^{\Delta}(t)=p(t) x(t), \qquad x(t_0)=1.
\end{equation*}
\end{theorem}


\section{Main results}
\label{sec:3}

We define the Caputo--Fabrizio (CF, for short)
fractional delta derivatives 
on time scales of order $\alpha\in[0,1)$:
left-sided and right-sided, with the scalar exponential function of time scales
as the kernel. For brevity, we restrict ourselves to the delta calculus.
However, similar definitions and results follow easily for the nabla case
by using the duality of fractional calculus and time scales
\cite{MR3662975,MyID:307,MR2957726}. Analogously, other
new notions of fractional delta derivatives
with respect to the matrix exponential function, as well
as higher-order derivatives, are also introduced.

\begin{definition}
Let $\T$ be a time scale, $\alpha\in (q-1,q]$,
where $q\in\N_0$, and $a,b\in \R$ with $a<b$.
We denote by $CF^\alpha[a,b]$ the class of functions
$f:[a,b]\rightarrow \R$ for which both delta integrals
$ \int_a^t f^{\Delta^q}(\tau) e_{\bar{\beta}}(t, \sigma(\tau))\Delta \tau$
and $ \int_t^b f^{\Delta^q}(\tau) e_{\bar{\beta}}(t, \sigma(\tau))\Delta \tau$,
where $\bar{\beta}=\frac{\beta}{\beta-1}$, $\beta=\alpha -q+1$, exist.
\end{definition}


\subsection{CF fractional delta derivative of order $\alpha \in [0,1)$}
\label{subsec:3.1}

As usual in fractional calculus, we introduce two notions
of fractional derivatives: left-sided and right-sided.

\begin{definition}[Caputo--Fabrizio fractional delta derivatives on time scales]
Let $\mathbb{T}$ be a time scale, $\alpha \in [0,1)$, $f\in CF^\alpha[a,b]$,
and $M$ a given function satisfying $M(0) = M(1) = 1$.
The left-sided Caputo--Fabrizio fractional delta derivative is defined by
\begin{equation}
\label{def:Caputo_onTS_l}
\left(_a^{CF}\Delta_t^{(\alpha)} f\right)(t)
:=\frac{M(\alpha)}{1 - \alpha} \int_a^t f^{\Delta}(\tau)
e_{\bar{\alpha}}(t, \sigma(\tau))\Delta \tau
\end{equation}
and the  right-sided Caputo--Fabrizio fractional delta derivative is defined by
\begin{equation}
\label{def:Caputo_onTS_r}
\left(_t^{CF}\Delta_b^{(\alpha)} f\right)(t)
:=\frac{M(\alpha)}{1 - \alpha} \int_t^b f^{\Delta}(\tau)
e_{\bar{\alpha}}(t, \sigma(\tau))\Delta \tau,
\end{equation}
where $\bar{\alpha} = \frac{\alpha}{\alpha-1}$.
\end{definition}

\begin{remark}
For $a=0$, the left-sided Caputo--Fabrizio fractional delta derivative
defined by \eqref{def:Caputo_onTS_l} has the form of convolution on a time scale:
\begin{equation*}
\left(_0^{CF}\Delta_t^{(\alpha)} f\right)(t)
=\frac{M(\alpha)}{1 - \alpha} \left(f^\Delta * e_{\bar{\alpha}} \right)(t).
\end{equation*}
\end{remark}

\begin{example}
When $\mathbb{T} = h\mathbb{Z}$, we get a new notion of fractional
difference. Namely, for $a,t\in h\mathbb{Z}$, we have
\begin{equation*}
\left(_a^{CF}\Delta_t^{(\alpha)} f\right)(t)
=\frac{M(\alpha)}{1 - \alpha} \sum_{k=\frac{a}{h}}^{\frac{t}{h}-1}
h f^{\Delta}(kh) \left(1+h\bar{\alpha} \right)^{\frac{t}{h}-k-1}.
\end{equation*}
\end{example}

\begin{lemma}
\label{transform}
Let $\alpha\in [0,1)$ and $F(z)=\mathcal{L}\left\{f\right\}(z)$. Then,
\begin{equation*}
\mathcal{L}\left\{^{CF}\Delta_t^{(\alpha)} f\right\}(z)
=M(\alpha)\frac{zF(z)-f(0)}{(1-\alpha)z+\alpha}.
\end{equation*}
\end{lemma}

\begin{proof}
Follows by adapting the proofs of the Laplace transform
of convolution on time scales found in \cite{Bohner_Lap}
and \cite{Davis_Lap}.
\end{proof}

\begin{proposition}
Let $f\in FC^{\alpha}[a,b]$ for $\alpha \in [0,1)$.
\begin{enumerate}
\item If $\alpha=0$, then the left-sided Caputo--Fabrizio
fractional delta derivative \eqref{def:Caputo_onTS_l}
becomes $\left(\mbox{}_a^{CF}\Delta_t^{(0)} f \right)(t)= f(t)-f(a)$
and $\left(_t^{CF}\Delta_b^{(0)} f \right)(t)= f(b)-f(t)$.

\item If $\alpha$ tends to 1 from the left side, then the Caputo--Fabrizio
fractional delta derivatives \eqref{def:Caputo_onTS_l} and \eqref{def:Caputo_onTS_r}
tend to the usual delta derivative of time scales,
that is, $\lim\limits_{\alpha\rightarrow 1^-}\left(\mbox{}_a^{CF}
\Delta_t^{(\alpha)} f\right) (t)= f^\Delta(t)$
and $\lim\limits_{\alpha\rightarrow 1^-}\left(\mbox{}_t^{CF}
\Delta_b^{(\alpha)} f\right)(t)=f^\Delta(t)$.
\end{enumerate}
\end{proposition}

\begin{proof}
The first part for $\alpha=0$ is obvious. For the second part,
we compute the limits and make use
of Lemma~\ref{transform}.
\end{proof}


\subsection{CF delta integrals of order $\alpha \in [0,1)$}
\label{subsec:3.3}

From now on, let $\T$ be an unbounded time scale.
In this section, using Laplace transform methods on time scales,
we present the solution to linear scalar equations.

Let $\alpha\in (0,1)$. Consider the following
delta-fractional differential equation:
\begin{equation}
\label{non1}
\left(^{CF}\Delta_t^{(\alpha)} x\right)(t)=u(t),
\quad t\in\T.
\end{equation}
Using the Laplace transform on time scales, we easily get
\begin{equation*}
x(t)=x(0)+\frac{1-\alpha}{M(\alpha)}(u(t)-u(0))
+\frac{\alpha}{M(\alpha)}\int_0^t u(\tau)\Delta \tau,
\quad t\geq 0.
\end{equation*}
It also means that any function defined as
\begin{equation*}
y(t)= c+\frac{1-\alpha}{M(\alpha)}u(t)+\frac{\alpha}{M(\alpha)}
\int_0^t u(\tau)\Delta \tau, \quad t\geq 0,
\end{equation*}
where $c\in\R$ is a constant, is also a solution of \eqref{non1}.
We can rewrite the delta-fractional equation \eqref{non1} as
\begin{equation*}
\int_0^t e_{\bar{\alpha}}(t,\sigma(\tau))x^{\Delta}(\tau)\Delta \tau
=\frac{1-\alpha}{M(\alpha)}u(t), \quad t\geq 0.
\end{equation*}
Moreover, as $e_{\bar{\alpha}}(t,\sigma(\tau))
=e_{\bar{\alpha}}(t,0)e_{\bar{\alpha}}(0,\sigma(\tau))$
and $e_{\bar{\alpha}}(0,t)=e^{-1}_{\bar{\alpha}}(t,0)$,
we can further write that
\begin{equation*}
\int_0^t e_{\bar{\alpha}}(0,\sigma(\tau))x^{\Delta}(\tau)\Delta \tau
=\frac{1-\alpha}{M(\alpha)} e_{\bar{\alpha}}(0,t)u(t),
\quad t\geq 0.
\end{equation*}
Then, taking the delta-derivative on both sides, we have
\begin{equation*}
e_{\bar{\alpha}}(0,\sigma(t))x^{\Delta}(\tau)
=\frac{1-\alpha}{M(\alpha)}
\left( e_{\bar{\alpha}}(0,t)u(t) \right)^\Delta.
\end{equation*}
As $\left( e_{\bar{\alpha}}(0,t)u(t) \right)^\Delta
=e^\Delta_{\bar{\alpha}}(0,t)u(t)
+e_{\bar{\alpha}}(0,\sigma(t))u^\Delta(t)$, then
\begin{equation*}
e_{\bar{\alpha}}(0,\sigma(t))x^{\Delta}(\tau)=\frac{1-\alpha}{M(\alpha)}
\left(e^\Delta_{\bar{\alpha}}(0,t)u(t)
+e_{\bar{\alpha}}(0,\sigma(t))u^\Delta(t)\right)
\end{equation*}
and
\begin{equation*}
x^{\Delta}(\tau)=\frac{1-\alpha}{M(\alpha)}
\frac{e^\Delta_{\bar{\alpha}}(0,t)}{e_{\bar{\alpha}}(0,\sigma(t))}u(t)
+\frac{1-\alpha}{M(\alpha)}u^\Delta(t).
\end{equation*}
Using the properties of delta derivatives on time scales,
direct calculations show that
\begin{equation*}
\frac{e^\Delta_{\bar{\alpha}}(0,t)}{e_{\bar{\alpha}}(0,\sigma(t))}
=-\bar{\alpha}=\frac{\alpha}{1-\alpha}.
\end{equation*}
Hence,
\begin{equation*}
x^{\Delta}(\tau)=\frac{\alpha}{M(\alpha)} u(t)
+\frac{1-\alpha}{M(\alpha)}u^\Delta(t),
\end{equation*}
which, after delta integration, gives
\begin{equation*}
x(t)=x(0)+\frac{\alpha}{M(\alpha)}\int_0^t u(\tau)\Delta \tau
+\frac{1-\alpha}{M(\alpha)}\left(u(t)-u(0)\right),
\quad t\geq 0.
\end{equation*}
Hence, similarly as in \cite{Nieto}, we can formulate
the fractional delta-integral of Caputo--Fabrizio type as follows.

\begin{definition}
Let $\alpha\in(0,1)$. The fractional delta-integral of order
$\alpha$ of a function $u$ is defined by
\begin{equation*}
\left(^{CF}I^\alpha u \right)(t)
:=\frac{1-\alpha}{M(\alpha)}u(t)+\frac{\alpha}{M(\alpha)}
\int_0^t u(\tau)\Delta \tau,
\quad t\geq 0.
\end{equation*}
\end{definition}

We can also consider, as an extension of \cite{Nieto},
that the Caputo--Fabrizio fractional delta-integral of order $\alpha$
of a function $u$ is an average between function $u$
and its delta integral of order one:
\begin{equation*}
\frac{1-\alpha}{M(\alpha)}+ \frac{\alpha}{M(\alpha)}=1.
\end{equation*}
From this, we have $M(\alpha)=1$ and formula \eqref{def:Caputo_onTS_l}
in the definition of left-sided Caputo--Fabrizio fractional derivative
on time scales started at $t_0=0$ can be stated as
\begin{equation*}
\left(^{CF}\Delta_t^{(\alpha)} f\right)(t)
:=\frac{1}{1 - \alpha} \int_0^t f^{\Delta}(\tau)
e_{\bar{\alpha}}(t, \sigma(\tau))\Delta \tau,
\end{equation*}
where $\bar{\alpha} = \frac{\alpha}{\alpha-1}$.


\subsection{Solutions of linear equations with CF fractional delta derivatives and their stability}
\label{subsec:3.2}

Now, let us consider the following linear scalar equation:
\begin{equation}
\label{scalar}
\left(^{CF}\Delta_t^{(\alpha)} x\right)(t)=\lambda x(t)+u(t),
\quad \lambda \in \R, \quad t\in\T,
\end{equation}
with initial condition $x(0)=x_0\in\R$. For simplicity,
we use $t_0=0$, which is assumed to belong to $\T$.
Here, we also use $M(\alpha)=1$, accordingly 
to results from Section~\ref{subsec:3.3}.

\begin{definition}
Let $\alpha\in [0,1)$. We say that $\lambda\in\R$ is 
$CF(\alpha)$-fractionally regressive
if $K(\alpha):=1- \lambda(1-\alpha)\neq 0$.
\end{definition}

\begin{theorem}
Let $\alpha\in [0,1)$ and $\T$ be an unbounded time scale
with $t_0=0\in\T$, $K(\alpha)\neq 0$. Then, equation
\eqref{scalar} subject to the initial condition $x(0)=x_0\in\R$
has the unique solution
\begin{equation}
\label{solution} 
\begin{split}
x(t)&= x(0)-\frac{1}{K(\alpha)}\left(1-e_{p(\alpha)}(t,0)\right)x(0)
+\frac{1-\alpha}{K(\alpha)}(u(t)-u(0))\\ 
& \qquad + \frac{\alpha}{K^2(\alpha)}\int_0^t
e_{p(\alpha)}(t,\sigma(\tau))u(\tau)\Delta \tau,
\end{split}
\end{equation}
where $p(\alpha)=\frac{\lambda\alpha}{K(\alpha)}$.
\end{theorem}

\begin{proof}
Using the Laplace transform method,
its inverse, and the methods of \cite{Bohner_Lap}, we obtain 
the formula of solution as 
\begin{equation}
\label{sol1}
x(t)=\frac{1}{K(\alpha)}e_{p(\alpha)}(t,0)x(0)+\frac{1-\alpha}{K(\alpha)}u(t)
+  \frac{\alpha}{K^2(\alpha)}\int_0^t
e_{p(\alpha)}(t,\sigma(\tau))u(\tau)\Delta \tau,
\end{equation}
while the formula
\begin{equation}
\label{sol2}
\tilde{x}(t)=\frac{1}{K(\alpha)}e_{p(\alpha)}(t,0)x(0)
+\frac{1-\alpha}{K(\alpha)}u(t)
+ \frac{\alpha}{K^2(\alpha)}\int_0^t
e_{p(\alpha)}(t,\sigma(\tau))u(\tau)\Delta \tau + C
\end{equation}
is also a solution of equation
\eqref{scalar} subject to the initial condition $x(0)=x_0$. 
To have uniqueness, we need to calculate 
$C=\left(1-\frac{1}{K(\alpha)}\right)x(0)-\frac{1-\alpha}{K(\alpha)}u(0)$, 
which together with equation \eqref{sol2} agrees with the form 
given by \eqref{solution}.
\end{proof}

\begin{corollary}
Let $\alpha\in [0,1)$ and $x(\cdot)$ be a solution of equation \eqref{scalar}
with $\lambda=0$, i.e., of equation  $\left(^{CF}\Delta_t^{(\alpha)} x\right)(t)
=u(t)$, $t\in \T$, subject to the initial condition $x(0)=x_0\in\R$. Then,
$\lambda$ is $CF(\alpha)$-fractionally regressive for any $\alpha \in [0,1)$
and $x(t)=x_0+(1-\alpha)(u(t)-u(0))+\alpha\displaystyle \int_0^tu(\tau)\Delta\tau$.
Moreover, for $u(t)\equiv 0$, we have the constant function $x(t) \equiv x_0$ for any $\alpha$.
\end{corollary}

\begin{remark}
Observe that solutions to equation \eqref{scalar} 
are continuous/bounded, if $u(\cdot)$ is a continuous/bounded function.
\end{remark}

Based on the results of P\"otzsche et al. \cite{Christian},
we make an analysis of the exponential stability of equation 
\eqref{scalar} in case of a constant function $u(\cdot)$. 
In this case, in the formula of solution \eqref{sol1}, 
it disappears the part $u(t)-u(0)$. In view of the definitions given in \cite{Christian}, 
the following notations for describing the sets of exponential stability are introduced.
For simplicity, we take $t_0=0\in\T$.

\begin{definition}[See \cite{Christian}]
Given a time scale $\T$ unbounded from above, we denote
\begin{equation*}
{S}_\mathbb{C}(\mathbb{T}):=\left\{\lambda\in\mathbb{C} :
\limsup_{T\to \infty} \frac{1}{T-t_0} \int\limits_{t_0}^T
\lim\limits_{s \to \mu(t)}\frac{\log | 1+s\lambda |}{s} \, 
\Delta t <0\right\}
\end{equation*}
and
\begin{equation*}
{S}_\mathbb{R}(\mathbb{T}):=\left\{\lambda\in\mathbb{R} :\ 
\forall T\in\mathbb{T}\ \ \exists t\in\mathbb{T},\  t>T \ :\
1+\mu(t)\lambda=0\right\}.
\end{equation*}
The set of exponential stability for the time scale $\T$ is then defined by
\begin{equation*}
{S}(\mathbb{T}):=
{S}_\mathbb{C}(\mathbb{T})\cup {S}_\mathbb{R}(\mathbb{T}).
\end{equation*}
\end{definition}

\begin{remark}
For an arbitrary time scale $\mathbb{T}$, we have
\(S_\mathbb{R} (\mathbb{T}) \subset (-\infty, 0)\)
and
\(S_\mathbb{C} (\mathbb{T}) \subseteq \left\{\lambda\in \mathbb{C}
: \ Re \, \lambda<0 \right\}\).
\end{remark}

\begin{proposition}
\label{thm:exp:stab}
Let $\T$ be unbounded from above, $p(\alpha)$ be regressive and 
$u(\cdot)$ be a constant function. Then, equation \eqref{scalar} 
is exponentially stable if and only if $p(\alpha)\in S_C(\T)$.
\end{proposition}

\begin{proof}
The result is obtained following the proof of \cite[Theorem~21]{Christian}.
\end{proof}

\begin{example}
Let $\alpha\in (0,1)$ and $M(\alpha)=1$.
We consider here two classical time scales.
\begin{enumerate}
\item[(1)] Let $\mathbb{T}=h\mathbb{Z}$, where $h>0$.
Then
\(
{S}_\mathbb{R}(h\mathbb{Z})=\left\{-\frac{1}{h}\right\}
\)
and
\(
{S}(h\mathbb{Z})=\mathcal{B}_{\frac{1}{h}}\left(-\frac{1}{h}\right)
\),
where  $\mathcal{B}_{\frac{1}{h}}\left(-\frac{1}{h}\right)$
denotes the disc with center at $\left(-\frac{1}{h},0\right)$
and radius $\frac{1}{h}$. Then, the stability region for
parameter $p(\alpha)$ is $S(h\Z)=B_{\frac{1}{h}}\left(-\frac{1}{h}\right)$.
Hence, it follows from Proposition~\ref{thm:exp:stab} that for real values 
of $p(\alpha)$, a necessary and sufficient condition for the exponential
stability of equation \eqref{scalar} is
\begin{itemize}
\item[(a)] $\lambda\in\left(-\frac{2}{h\alpha-2(1-\alpha)},0\right)$,
for $h>2\left(\frac{1}{\alpha}-1\right)$;

\item[(b)] $\lambda\in (-\infty, 0)\cup \left(\frac{2}{2(1-\alpha)
-h\alpha},+\infty\right)$, for $h\leq 2\left(\frac{1}{\alpha}-1\right)$.
\end{itemize}
Observe that for $\alpha$ tending to $1$ only item (a)
is possible, and then the condition agrees with
$\lambda\in \left(-\frac{2}{h},0\right)$.

\item[(2)] Let $\mathbb{T}=\mathbb{R}$. Then
\(
{S}_\mathbb{R}(\mathbb{R})=\emptyset
\)
and
\(
{S}(\mathbb{R})=\{\lambda\in\mathbb{C}:\ Re\lambda<0\}.
\)
Equation \eqref{scalar} is exponentially stable if and only if
$\lambda <0$ or $\lambda>\frac{1}{1-\alpha}$.
\end{enumerate}
\end{example}

\begin{example}
Let $\alpha\in (0,1]$,
$K(\alpha):=1- \lambda(1-\alpha)\neq 0$,
and consider the initial condition $x(0)=0\in\T$. 
Then, the unique solution of the initial value problem
has the form
\begin{equation}
\label{ex1}
x(t)=\frac{1-\alpha}{K(\alpha)}(u(t)-u(0))
+\frac{\alpha}{K^2(\alpha)}\int_0^t
e_{p(\alpha)}(t,\sigma(\tau))u(\tau)\Delta \tau
\end{equation}
with $p(\alpha)=\frac{\lambda\alpha}{K(\alpha)}$.
Considering $\T=h\Z$, $t=kh$, we have that
$$
\int_0^t
e_{p(\alpha)}(t,\sigma(\tau))u(\tau)\Delta \tau
=h\sum_{s=0}^{k-1}\left(1+hp(\alpha)\right)^{k-s-1}u(sh)
$$
holds. We present graphs with the behavior of the exact 
solution in different situations:
\begin{itemize}
\item[(a)] For $\T=\Z$, $\lambda=0.2$, $u(t)\equiv 1$,  
and $\alpha\in\{0.2,0.5,0.9,1\}$, we get different exponents 
tending to infinity, what we see for first 30 steps in Figure~\ref{rys1}. 
The solutions are unstable.
\begin{figure}[ht!]
\centering
\includegraphics[scale=0.6]{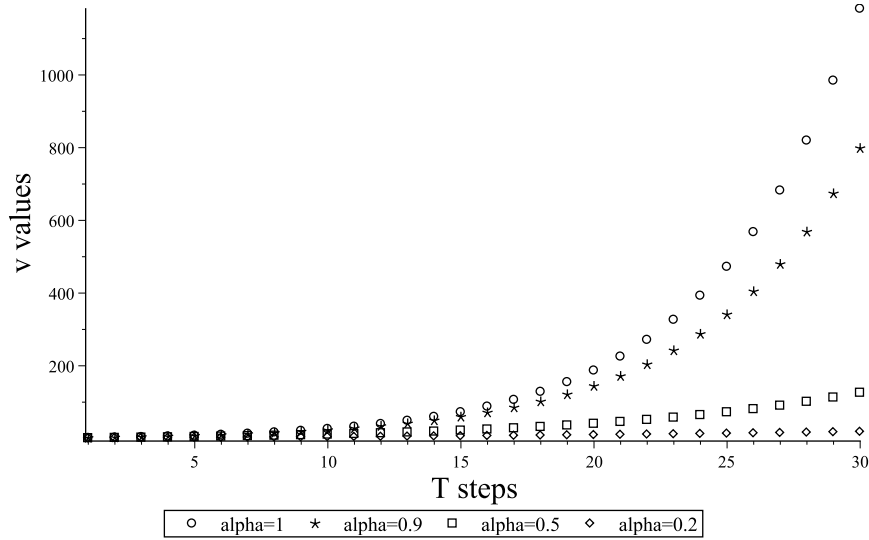}
\caption{Graphs of solutions to the linear equation \eqref{scalar} with  
$\lambda=0.2$, $x(0)=0$, $u(t)\equiv 1$, given by \eqref{ex1},  
for $\alpha\in\{0.2,0.5,0.9,1\}$ on $\T=\Z$, for first 30 steps.
\label{rys1}}
\end{figure}
\item[(b)] If we would like to receive stable solutions, then 
we can take, for $\T=\Z$, negative $\lambda$ or big enough $\lambda$. 
Taking $\lambda=4.2$ we have stable solutions for $\alpha\in\{0.2,0.5\}$ 
while for $\alpha>4$ the solution is stable for $\alpha=0.5$. If $u(t)\equiv 1$, 
then there is now a first part in summation in the solution's formula \eqref{ex1}. 
In this case, the values of the control function $u(\cdot)$ do not influence the stability.  
The situation is presented in Figure~\ref{rys2}.
\begin{figure}[ht!]
\centering
\includegraphics[scale=0.6]{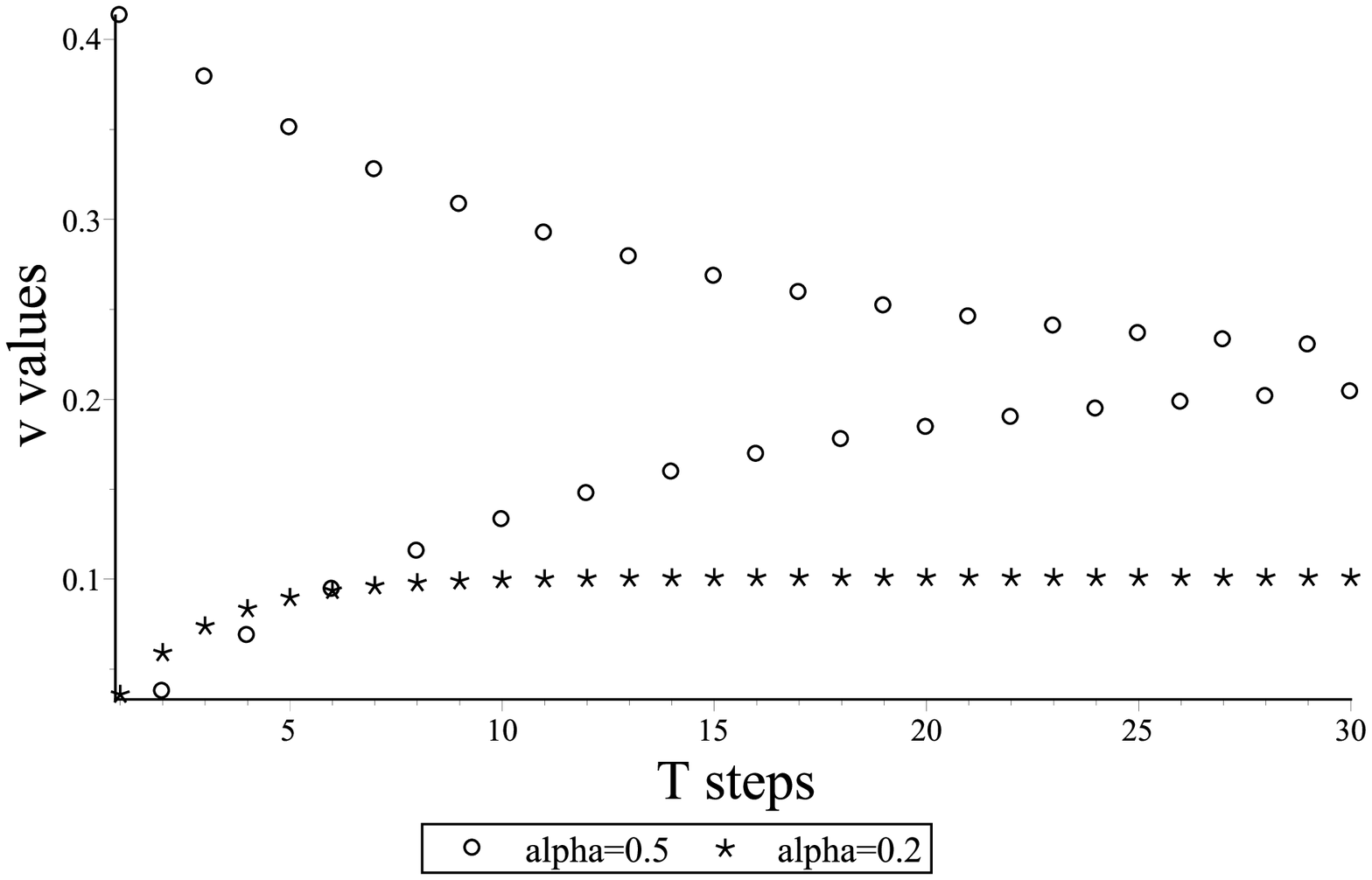}
\caption{Graphs of solutions to the linear equation \eqref{scalar} with  
$\lambda=4.2$, $x(0)=0$, $u(t)\equiv 1$, given by \eqref{ex1},  
for $\alpha\in\{0.2,0.5\}$ on $\T=\Z$, for first 30 steps.
\label{rys2}}
\end{figure}
\item[(c)] In Figure~\ref{rys3}, we compare graphs with $\lambda=0.2$ 
for different steps $h\in\{0.1,0.5,1\}$ and  we take order $\alpha=0.5$. 
If we include the situation when $h \rightarrow 0$, which in the limit 
represents $\T=\R$, we will receive the line on dots for $h=0.1$.
\begin{figure}[ht!]
\centering
\includegraphics[width=0.80\textwidth]{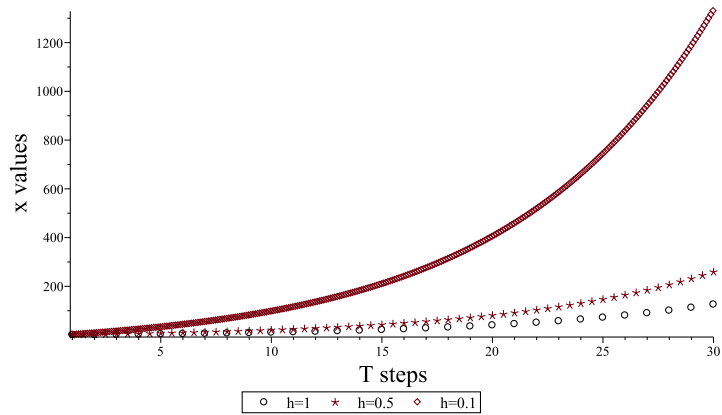}
\caption{Graphs of solutions to the linear equation \eqref{scalar} with  
$\lambda=0.2$, $x(0)=0$, given by \eqref{ex1}, with 
$u(t)\equiv 1$ for $h\in\{0.1,0.5,1\}$ and $\alpha=0.5$, for first 30 steps.
\label{rys3}}
\end{figure}
\end{itemize}
\end{example}


\subsection{Nonlinear CF delta-fractional differential equations}
\label{subsec:3.4}

Under suitable assumptions on function $f(\cdot,x(\cdot))$ of equation
\begin{equation}
\label{non2}
\left(^{CF}_a\Delta_t^{(\alpha)} x\right)(t)=f(t,x(t)),
\quad x(a)=x_0\in \R, \quad  t\in [a,b]\cap \T,
\end{equation}
we can extend our results to the nonlinear case
by using a fixed point theorem.

\begin{definition}[Cf. \cite{Bohner}]
Let $\T$ be a time scale. A function $f: \T\times \R\rightarrow\R$ is called
\begin{itemize}
\item[(i)] rd-continuous, if $g$ defined by $g(t)=f(t,x(t))$ is rd-continuous
for any continuous function $x:\T\rightarrow \R$;

\item[(ii)] regressive at $t\in\T^{\kappa}$, if the mapping
$1+\mu(t)f(t,\cdot):\R\rightarrow \R$ is invertible;
and $f$ is called regressive on $\T^\kappa$,
if $f$ is regressive at each $t\in\T^\kappa$;

\item[(iii)] bounded on a set $S\subset \T\times \R$,
if there exists a constant $M>0$ such that
$|f(t,x)|\leq M$ for all $(t,x)\in S$;

\item[(iv)] Lipschitz continuous on a set $S\subset \T\times \R$,
if there is a constant $L>0$ such that $\left|f(t,x_1)-f(t,x_2) \right|
\leq L\left| x_1-x_2\right|$ for all $(t,x_1), (t,x_2)\in S$.
\end{itemize}
\end{definition}

\begin{theorem}
Let $\alpha\in(0,1)$, $b>a$, and $f:[a,b]\cap\T\times \R\rightarrow \R$
be a rd-continuous function such that there exists $L>0$ such that
\begin{equation*}
|f(t,x_1)-f(t,x_2)|\leq L|x_1-x_2|\, \ \ \mbox{for all} \  x_1,x_2\in \R.
\end{equation*}
If $((1-\alpha)+\alpha (b-a))L<1$, then the initial value problem
\eqref{non2} has a unique solution on $[a,b]\cap \T$.
\end{theorem}

\begin{proof}
Let $||f||=\sup\limits_{t\in[a,b]\cap \T}\left|f(t)\right|$ for all
rd-continuous functions on $[a,b]\cap \T$, that is,
for all $f \in \mathcal{C}_{rd}([a,b]_{\T})$.
Consider the operator $\mathcal{N}: \mathcal{C}_{rd}([a,b]_{\T})
\to\mathcal{C}_{rd}([a,b]_{\T})$ defined by
\begin{equation*}
\left(\mathcal{N} x\right) (t)=x(0)
+\alpha \int_a^t f(\tau,x(\tau))\Delta \tau
+(1-\alpha)\left(f(t,x(t))-f(0,x(0))\right),
\end{equation*}
$x\in \mathcal{C}_{rd}([a,b]_{\T})$.
Because for all $x_1,x_2\in \mathcal{C}_{rd}([a,b]_{\T})$
and all $t\in [a,b]_{\T}$ we have
\begin{equation*}
\begin{split}
& \left|\left(\mathcal{N} x_1\right) (t) -\left(\mathcal{N} x_2\right) (t)\right|\\
&\quad \leq (1-\alpha) \left| f(t,x_1(t))-f(t,x_2(t))\right| +\alpha \int_a^t
\left| f(\tau,x_1(\tau))-f(\tau,x_2(\tau))\right| \Delta \tau \\
&\quad \leq (1-\alpha)L||x_1-x_2||+\alpha L(b-a)||x_1-x_2||\\
&\quad \leq  \left(1-\alpha +\alpha(b-a) \right)L||x_1-x_2||,
\end{split}
\end{equation*}
we conclude that the operator $\mathcal{N}$ is a contraction.
The statement follows from Banach's fixed point theorem.
\end{proof}


\section{Conclusion}
\label{sec:conc}

The theory of fractional differential equations on time scales,
specifically the questions of existence, uniqueness of solutions
and their stability, is a recent research topic of great importance, 
see \cite{MyID:379,MyID:403} and references therein. Here, we studied 
deterministic fractional operators on time scales with Caputo--Fabrizio 
type kernels. The main difference is seen in solutions, where the 
language of integrals and exponentials on time scales is used.
Such non-singular kernels have recently been applied successfully to some
real world problems \cite{Thabet2,Abdon,Caputo:Fabrizio,CF1,Nieto}.
For this reason, we trust that the general concepts and calculus
here introduced will initiate further interest and developments.
As possible future work, motivated respectively by the recent
techniques and ideas of \cite{MR3614829} and \cite{MyID:365},
we mention: proof of chain rules and inequalities, and
existence and uniqueness results of \emph{positive solutions}.


\subsubsection*{Acknowledgments}

The work of D.~Mozyrska and M.~Wyrwas was supported by 
Bialystok University of Technology, grant
S/WI/1/2016, and funded by the resources for research 
of Polish Ministry of Science and Higher Education.
The work of D.~F.~M.~Torres has been partially supported
by CIDMA and FCT within project UID/MAT/04106/2013.
The authors are grateful to an Associated Editor 
and three anonymous Reviewers for thorough comments, 
which showed them how to modify the manuscript 
in order to improve its quality.
 

\small




\begin{thebibliography}{99}

\bibitem{Thabet2}
T. Abdeljawad and D. Baleanu,
On fractional derivatives with exponential kernel and their discrete versions,
\emph{Rep. Math. Phys.} {\bf 80} (1), 2017, 11--27.

\bibitem{MR3662975}
T. Abdeljawad and D. F. M. Torres,
Symmetric duality for left and right Riemann-Liouville and Caputo fractional differences,
\emph{Arab J. Math. Sci.} {\bf 23} (2), 2017, 157--172.
{\tt arXiv:1607.03458}

\bibitem{MR3776052}
T. M. Atanackovi\'c, S. Pilipovi\'c\ and\ D. Zorica,
Properties of the Caputo-Fabrizio fractional derivative and its distributional settings,
\emph{Fract. Calc. Appl. Anal.} {\bf 21} (1), 2018, 29--44.

\bibitem{ATANGANA2018688}
A. Atangana, 
Non validity of index law in fractional calculus: a fractional 
differential operator with Markovian and non-Markovian properties, 
\emph{Phys. A} {\bf 505}, 2018, 688--706.

\bibitem{Abdon}
A. Atangana\ and\ D. Baleanu,
New fractional derivatives with non-local and non-singular kernel:
theory and application to heat transfer model,
\emph{Therm. Sci.} {\bf 20} (2), 2016, 763--769.

\bibitem{Atangana2018}
A. Atangana\ and\ J. F. G{\'o}mez-Aguilar,
Decolonisation of fractional calculus rules: Breaking commutativity 
and associativity to capture more natural phenomena,
\emph{Eur. Phys. J. Plus} {\bf 133}, 2018, 1--23.

\bibitem{Aydogan2017}
M. S. Aydogan, D. Baleanu, A. Mousalou\ and\ S. Rezapour,
On approximate solutions for two higher-order Caputo-Fabrizio 
fractional integro-differential equations, 
\emph{Adv. Difference Equ.} {\bf 2017}, 2017, Paper no.~221, 11~pp.

\bibitem{Aydogan2018}
M. S. Aydogan, D. Baleanu, A. Mousalou\ and\ S. Rezapour,
On high order fractional integro-differential equations 
including the Caputo-Fabrizio derivative, 
\emph{Bound. Value Probl.} {\bf 2018}, 2018, Paper no.~90, 15~pp.

\bibitem{Baleanu_BVP2017}
D. Baleanu, A. Mousalou\ and\ S. Rezapour, 
On the existence of solutions for some infinite coefficient-symmetric 
Caputo-Fabrizio fractional integro-differential equations, 
\emph{Bound. Value Probl.} {\bf 2017}, 2017, Paper no.~145, 9~pp.

\bibitem{MyID:201}
N. R. O. Bastos, D. Mozyrska\ and\ D. F. M. Torres,
Fractional derivatives and integrals on time scales
via the inverse generalized Laplace transform,
\emph{Int. J. Math. Comput.} {\bf 11}, 2011, 1--9.
{\tt arXiv:1012.1555}

\bibitem{MyID:379}
B. Bayour, A. Hammoudi\ and\ D. F. M. Torres,
A truly conformable calculus on time scales,
\emph{Glob. Stoch. Anal.} {\bf 5} (1), 2018, 1--14.
{\tt arXiv:1705.08928}

\bibitem{MyID:330}
B. Bayour\ and\ D. F. M. Torres,
Complex-valued fractional derivatives on time scales,
in {\it Differential and difference equations with applications},
79--87, Springer Proc. Math. Stat., 164, Springer, Cham, 2016.
{\tt arXiv:1511.02153}

\bibitem{MyID:296}
N. Benkhettou, A. M. C. Brito da Cruz\ and\ D. F. M. Torres,
A fractional calculus on arbitrary time scales:
Fractional differentiation and fractional integration,
\emph{Signal Process.} {\bf 107}, 2015, 230--237.
{\tt arXiv:1405.2813}

\bibitem{MyID:320}
N. Benkhettou, A. M. C. Brito da Cruz\ and\ D. F. M. Torres,
Nonsymmetric and symmetric fractional calculi on arbitrary nonempty closed sets,
\emph{Math. Methods Appl. Sci.} {\bf 39} (2), 2016, 261--279.
{\tt arXiv:1502.07277}

\bibitem{MyID:328}
N. Benkhettou, A. Hammoudi\ and\ D. F. M. Torres,
Existence and uniqueness of solution for a fractional
Riemann--Liouville initial value problem on time scales,
\emph{J. King Saud Univ. Sci.} {\bf 28} (1) 2016, 87--92.
{\tt arXiv:1508.00754}

\bibitem{MyID:324}
N. Benkhettou, S. Hassani\ and\ D. F. M. Torres,
A conformable fractional calculus on arbitrary time scales,
\emph{J. King Saud Univ. Sci.} {\bf 28} (1), 2016, 93--98.
{\tt arXiv:1505.03134}

\bibitem{Bohner}
M. Bohner\ and\ A. Peterson,
{Dynamic Equations on Time Scales: An Introduction with Applications}, 2001,
Birkh\"auser Boston, Boston, MA, USA.

\bibitem{Bohner_Lap}
M. Bohner\ and\ A. Peterson,
Laplace transform and $Z$-transform: unification and extension,
\emph{Methods Appl. Anal.} {\bf 9} (1), 2002, 151--157.

\bibitem{Caputo:Fabrizio}
M. Caputo\ and\ M. Fabrizio,
A new definition of fractional derivative without singular kernel,
\emph{Progr. Fract. Differ. Appl.} {\bf 1} (2), 2015, 73--85.

\bibitem{CF1}
M. Caputo\ and\ M. Fabrizio,
Applications of new time and spatial fractional derivatives
with exponential kernels,
\emph{Progr. Fract. Differ. Appl.} {\bf 2} (1), 2016, 1--11.

\bibitem{MyID:307}
M. C. Caputo\ and\ D. F. M. Torres,
Duality for the left and right fractional derivatives,
\emph{Signal Process.} {\bf 107}, 2015, 265--271.
{\tt arXiv:1409.5319}

\bibitem{MR3770784}
J. M. Cruz-Duarte, J. Rosales-Garcia, C. R. Correa-Cely,
A. Garcia-Perez\ and\ J. G. Avina-Cervantes,
A closed form expression for the Gaussian-based Caputo-Fabrizio
fractional derivative for signal processing applications,
\emph{Commun. Nonlinear Sci. Numer. Simul.} {\bf 61}, 2018, 138--148.

\bibitem{Davis_Lap}
J. M. Davis, I. A. Gravagne, B. J. Jackson,
R. J. Marks \ and \ A. A. Ramos,
The Laplace transform on time scales revisited,
\emph{J. Math. Anal. Appl.} {\bf 332}, 2007, 1291--1307.

\bibitem{MR3787701}
M. A. Firoozjaee, H. Jafari, A. Lia \ and \ D. Baleanu,
Numerical approach of Fokker--Planck equation with Caputo--Fabrizio
fractional derivative using Ritz approximation,
\emph{J. Comput. Appl. Math.} {\bf 339}, 2018, 367--373.

\bibitem{MR2957726}
E. Girejko\ and\ D. F. M. Torres,
The existence of solutions for dynamic inclusions
on time scales via duality,
\emph{Appl. Math. Lett.} {\bf 25} (11), 2012, 1632--1637.
{\tt arXiv:1201.4495}

\bibitem{Gomez-Aguilar2017}
J. F. G\'{o}mez-Aguilar, M. G. L\'opez-L\'opez, V. M. Alvarado-Mart\'{\i}nez, D.  Baleanu\ and\ H. Khan,
Chaos in a cancer model via fractional derivatives with exponential decay and Mittag-Leffler law, 
\emph{Entropy} {\bf 19}, 2017, no.~12, Paper no.~681, 19~pp.

\bibitem{Nieto}
J. Losada\ and\ J. J. Nieto,
Properties of a new fractional derivative without singular kernel,
\emph{Progr. Fract. Differ. Appl.} {\bf 1} (2), 2015, 87--92.

\bibitem{MyID:403}
K. Mekhalfi\ and\ D. F. M. Torres,
Generalized fractional operators on time scales with application to dynamic equations,
\emph{Eur. Phys. J. Special Topics} {\bf 226} (16-18), 2017, 3489--3499.
{\tt arXiv:1804.02536}

\bibitem{MR3614829}
E. R. Nwaeze\ and\ D. F. M. Torres,
Chain rules and inequalities for the BHT fractional calculus on arbitrary timescales,
\emph{Arab. J. Math.} (Springer) {\bf 6} (1), 2017, 13--20.
{\tt arXiv:1611.09049}

\bibitem{MyID:343}
M. D. Ortigueira, D. F. M. Torres\ and\ J. J. Trujillo,
Exponentials and Laplace transforms on nonuniform time scales,
\emph{Commun. Nonlinear Sci. Numer. Simul.} {\bf 39}, 2016, 252--270.
{\tt arXiv:1603.04410}

\bibitem{Christian}
C. P\"otzsche, S. Siegmund\ and\ F. Wirth,
A spectral characterization of exponential stability for linear
time-invariant systems on time scales,
\emph{Discrete Contin. Dyn. Syst.} {\bf 9} (5), 2003, 1223--1241.

\bibitem{MyID:365}
M. R. Sidi Ammi\ and\ D. F. M. Torres,
Existence and uniqueness results for a fractional Riemann-Liouville
nonlocal thermistor problem on arbitrary time scales,
\emph{J. King Saud Univ. Sci.} {\bf 30} (3), 2018, 381--385.
{\tt arXiv:1703.05439}

\bibitem{TateishiRibeiroLenzi2017}
A. A. Tateishi, H. V. Ribeiro\ and\ E.K. Lenzi,
The role of fractional time-derivative operators on anomalous diffusion, 
\emph{Front. Phys.} {\bf 5}, 2017, 9~pp.

\end{thebibliography}
\end{document}